\documentclass[11pt,oneside]{amsart}
\usepackage{amsmath,ifthen, amsfonts, amssymb, srcltx,
   amsopn, textcomp}
\usepackage{tikz-cd}
\usepackage{hyperref}
\usepackage{graphicx}
\graphicspath{ {C:\Users\13154\Desktop\Paper} }
\usepackage{tikz, tikz-cd}
\usepackage{graphicx}
\usepackage[small]{caption}

\newcommand{\showcomments}{yes}

\newcommand{\hidetodo}[1]
{\ifthenelse{\equal{\showcomments}{yes}}%
{#1}
}

\newsavebox{\commentbox}
{\ifthenelse{\equal{\showcomments}{yes}}%
{\footnotemark
        \begin{lrbox}{\commentbox}
        \begin{minipage}[t]{1.25in}\raggedright\sffamily\tiny
        \footnotemark[\arabic{footnote}]}
{\begin{lrbox}{\commentbox}}}%
{\ifthenelse{\equal{\showcomments}{yes}}%
{\end{minipage}\end{lrbox}\marginpar{\usebox{\commentbox}}}
{\end{lrbox}}}

\newtheorem{thm}{Theorem}[section]
\newtheorem{lem}[thm]{Lemma}

\newtheorem{cor}[thm]{Corollary}

\theoremstyle{definition}
\newtheorem{defn}[thm]{Definition}
\newtheorem{rem}[thm]{Remark}

\makeatletter
\newtheorem*{rep@theorem}{\rep@title}
\newcommand{\newreptheorem}[2]{%
\newenvironment{rep#1}[1]{%
 \def\rep@title{#2 \ref{##1}}%
 \begin{rep@theorem}}%
 {\end{rep@theorem}}}
\makeatother

\newreptheorem{theorem}{Theorem}

\newreptheorem{lemma}{Lemma}

\DeclareMathOperator{\kernel}{ker}

\DeclareMathOperator{\rank}{rank}

\newcommand{\semidirect}{\ensuremath{\rtimes}}

\newcommand{\homology}{\ensuremath{{\sf{H}}}}

\newcommand{\field}[1]{\mathbb{#1}}
\newcommand{\integers}{\ensuremath{\field{Z}}}

\newcommand{\naturals}{\ensuremath{\field{N}}}

\newcommand{\euler}{\chi}

\newcommand{\nclose}[1]{\ensuremath{\langle\!\langle#1\rangle\!\rangle}}

\begin{document}
\title[Failure of the FGIP for ascending HNN extensions of free groups]{Failure of the finitely generated intersection property for ascending HNN extensions of free groups}
\maketitle
\begin{center}
\author{JACOB BAMBERGER AND DANIEL T. WISE}
\end{center}

\section{Abstract}

The main result in this paper is the failure of the finitely generated intersection property (FGIP) of ascending HNN extensions of non-cyclic finite rank free groups. This class of group consists of free-by-cyclic groups and properly ascending HNN extensions of free groups. We also give a sufficient condition for the failure of the FGIP in the context of relative hyperbolicity, we apply this to free-by-cyclic groups of exponential growth.

\section{Introduction}

\begin{defn} A group has the \textit{finitely generated intersection property} (FGIP) if the intersection of any two finitely generated subgroups is also finitely generated.
\end{defn}
The most famous class of groups having the FGIP are locally quasiconvex word-hyperbolic groups \cite{Short91}. This generalizes the fact that free groups have the FGIP which was proven by Howson \cite{Howson54}, and indeed the FGIP is sometimes referred to as the \textit{Howson property}.

The purpose of this paper is to examine the FGIP for ascending HNN extensions of finitely generated free groups. We show the failure of the FGIP for ascending HNN extensions of non-cyclic finite rank free groups. 
\begin{defn}[Ascending HNN extension of a free group] \label{def:ascHNN}
Let $\phi: F \rightarrow F$ be a monomorphism from a free group to itself. Its associated \textit{ascending HNN extension} is the group $G=\langle F, t \mid tft^{-1}=\phi(f) \ : \ \forall f \in F \rangle$. If $\phi$ is surjective then $G$ is \textit{free-by-cyclic} and we denote $G$ by $F \rtimes_\phi \integers$. If $\phi$ is not surjective then $G$ is a \textit{proper ascending HNN extension}.\end{defn}

Our main goal is the following result which is new for proper HNN extensions:

\begin{thm}
Any ascending HNN extension $G$ of a finitely generated free group $F$ of rank $\geq 2$ fails to have the FGIP.
\label{maintheorem}\end{thm}

The proof of Theorem~\ref{maintheorem} is partitioned into Theorem~\ref{properHNN}, Theorem~\ref{polynomialcase}, and Corollary~\ref{freebchyp}. This generalizes the following result of Berns and Brunner \cite{BurnsRobBrun} who generalized Moldavanskii's observation that the FGIP fails for $F \times \integers$  \cite{Moldavanski68}.

\begin{thm}[Berns-Brunner]
Let $G=F\semidirect\integers$ where $F$ is a finitely generated free group of rank $\geq 2$. Then $G$ does not have the FGIP.
\label{burnsandbrunner}
\end{thm}

Finitely generated subgroups of ascending HNN extensions of rank $1$ free groups are easily shown to be either trivial, cyclic or of finite index (see e.g. \cite{LiWise2017}). Intersecting a subgroup $H$ with a finite index subgroup $K$ yields a finite index subgroup of $H$. Hence if $H$ is finitely generated so is $H\cap K$. We therefore focus on ascending HNN extensions of a free group $F$ with rank$(F)\geq2$.

The main principle here is that a nontrivial infinite index normal subgroup of a free group cannot be finitely generated (see Corollary~\ref{normalfinite}). Thus if $G$ contains a finitely generated free group $K$ and a finitely generated normal subgroup $N$, the intersection $K\cap N$ provides an opportunity for the failure of the FGIP.
Theorem~\ref{maintheorem} exploits a generalization of the principle to subgroups that are conjugated into themselves as explained in Lemma~\ref{properconj}.

When $G = N \semidirect \integers$ is a hyperbolic group and $N$ is finitely generated, the failure of the FGIP for $G$ has a crisp explanation using the ping pong lemma (see Section~\ref{pingpongsubsect}). The idea is that for any $w\in N-\{1_G\}$, the ping pong lemma provides a free subgroup $K=\langle w^m , t^m \rangle$ and $K\cap N$ is not finitely generated. We develop this idea further in Section~\ref{relhyp} to obtain the following sufficient condition for the failure of the FGIP in the relative hyperbolic framework:
\begin{reptheorem}{relhypconj2}
Let $G$ be hyperbolic relative to a collection of subgroups. Let $N \subseteq G$ be a finitely generated subgroup containing a loxodromic element $w$. Suppose $tNt^{-1} \subseteq N$ for some infinite order $t$ with $\langle t \rangle  \cap N = \{1_G\}$.  Then $G$ fails to have the FGIP. In particular, there exists $m$ such that $\langle t^m, w^m \rangle \cap N$ is not finitely generated.
\end{reptheorem}

This applies to a non-elementary hyperbolic group $G$ arising as an ascending HNN extension of a group $N$. Hence Theorem~\ref{relhypconj2} generalizes the hyperbolic case of Theorems~\ref{maintheorem}~and~\ref{burnsandbrunner}.

We combine Theorem~\ref{relhypconj2} with recent results about relative hyperbolicity of free-by-cyclic groups \cite{dahmani2019relative, gautero2007mappingtorus, ghosh2018relative} to prove the failure of the FGIP for exponentially growing free-by-cyclic groups in Section~\ref{rexporelhyp}. This explanation is complex since it depends upon a constellation of deep results, but it is  interesting to see how the exponential case fits into the relative hyperbolic framework via Theorem~\ref{relhypconj2}. We expect this unity to prevail as relatively hyperbolic structures are constructed for general ascending HNN extensions.

In Section~\ref{properascendingsection} we prove the failure of the FGIP for proper ascending HNN extensions of finite rank free groups. In Section~\ref{traintracksection} we introduce train track maps. The theory of relative train track maps explains that automorphisms of free groups are divided into polynomially growing and exponentially growing cases \cite{Bestvina_2000, traintracks}. In Section~\ref{polycase} we use relative train track maps to explain the FGIP failure for $F\semidirect\integers$ in the polynomially growing case. This exploits their structure as multiple cyclic HNN extensions. In Section~\ref{rexporelhyp} we apply Theorem~\ref{relhypconj2} to prove that the FGIP fails for $F\semidirect\integers$ in the exponentially growing case. Note that the exponentially growing case is the main case since failure in the polynomially growing case is a simple consequence of the failure for $F \times \integers$. While our argument is more complex than Theorem~\ref{burnsandbrunner} (which we were originally oblivious to), we hope it will appeal to the contemporary reader for whom train tracks and relative hyperbolicity are established tools.

\section{Proper ascending HNN extension}
\label{properascendingsection}
In this section we prove the following theorem:

\begin{thm}
Any proper ascending HNN extension of a non-cyclic finitely generated free group fails to have the FGIP. \label{properHNN}
\end{thm}


We begin by recalling a Theorem about free groups first proved by Hall \cite{MR28836}:

\begin{thm}
If $F$ is a free group and $H$ is a finitely generated subgroup then $H$ is a free factor in a finite-index subgroup of $F$. \label{thm:MarshallHall}
\end{thm}

\begin{cor}
Let $H \subsetneq F$ be a finitely generated infinite index subgroup of a free group. There exists a non-trivial $f$ such that $\langle f,  H \rangle=\langle f \rangle \ast H$.
\label{cor:splits}
\end{cor}
\begin{proof}
Theorem~\ref{thm:MarshallHall} says that some finite index subgroup $G \subseteq F$ satisfies $G=H\ast K$, so taking any non-trivial $f \in K$ concludes the proof.
\end{proof}

We now state and prove results that will be useful throughout the paper.

\begin{lem}
Let $\phi:G \rightarrow G$ be group automorphism. The following are equivalent for $w\in G$. 
\begin{enumerate}
    \item $\{\phi^{i}(w), \ i \in \mathbb{Z}\}$ form a basis for a free subgroup.
    \item $\{\phi^{i}(w), \ i \in \mathbb{N}\}$ form a basis for a free subgroup.
\end{enumerate}
\label{freeinf}
\end{lem}

\begin{proof}
    (1) $\Rightarrow$ (2) is clear, so we focus on (2) $\Rightarrow$ (1). If a product $\prod_{i=1}^{k}\phi^{n_i}(w)$  where $n_i \in \integers$ represents the identity then $\phi^{max(|n_i|)}(\prod_{i=1}^{k}\phi^{n_i}(w))=e$, so $\prod_{i=1}^{k}\phi^{m_i}(w)$ where $m_i=max(|n_i|)+n_i \in \mathbb{N}$ represents the identity. 
\end{proof}

\begin{cor}
Let $G$ be a group and $t,w \in G$. The following are equivalent:
\begin{enumerate}
    \item $\{t^{i}wt^{-i}, \ i \in \mathbb{Z}\}$ form a basis for a free subgroup.
    \item $\{t^{i}wt^{-i}, \ i \in \mathbb{N}\}$ form a basis for a free subgroup.
\end{enumerate}
\label{freecor}
\end{cor}

\begin{proof}
Apply Lemma~\ref{freeinf} to the inner automorphism consisting of conjugation by the element $t$.
\end{proof}

\begin{lem} Let $H$ be a group generated by elements $w$ and $t$, and let $\rho: H \rightarrow \integers$ be a homomorphism with $\rho(t)=1$ and $\rho(w)=0$. Then 
$\ker(\rho)=\nclose w=\langle t^{i}wt^{-i}: i \in \integers \rangle$. 
In particular, if $\{t^{i}wt^{-i}, \ i \in \mathbb{Z}\}$ forms a basis for a free subgroup then $H$ is free of rank $2$.
\label{infgencor}
\end{lem}

We use the notation $\nclose w$ for the normal closure of $w$.


\begin{proof}
    
    Let $J$ be the free group on $w$ and $t$. We get a homomorphism $J\rightarrow \integers$ by composing $J \rightarrow H \rightarrow \integers$. Its kernel is the normal closure of $w \in J$, since $w$ is mapped to 0 and $J/\nclose w \cong \integers$. Realizing $J$ as the fundamental group of a bouquet of circles on $w$ and $t$ and considering the associated cyclic covering space one sees that the normal closure of $w$ in $J$ is generated by conjugates of $w$ by $t^i$ for $i \in \integers$.
    
    The kernel of $J \rightarrow \integers$ maps surjectively on the kernel of $H \rightarrow \integers$. The image under $J \rightarrow H$ of the normal closure of $w$ in $J$ is the normal closure of $w$ in $H$ and the image of conjugates of $w$ by $t^i$ in $J$ are conjugates of $w$ by $t^i$ in $H$, so $\ker(\rho)=\nclose w=\langle t^{i}wt^{-i}: i \in \integers \rangle$.



We have constructed the following communtative diagram:

\begin{center}
    \begin{tikzcd}
    1 \arrow[r] \arrow[d] & \kernel(\rho \circ \tau) \arrow[r , hook] \arrow[d] & J \arrow[r, "\rho \circ \tau", two heads] \arrow[d, "\tau"] & \mathbb{Z} \arrow[r] \arrow[d, "id"] & 1 \arrow[d] \\
    1 \arrow[r]           & \kernel(\rho) \arrow[r, hook]           & H \arrow[r, "\rho", two heads]                & \mathbb{Z} \arrow[r]                & 1               
    \end{tikzcd}
\end{center}

If $\{t^{i}wt^{-i}, \ i \in \mathbb{Z}\}$ form a basis for a free subgroup of $H$, then $\tau:J \rightarrow H$ restricts to an injection on the kernel of the homomorphisms to $\integers$ and is therefore an isomorphism between the kernels. Conclude by applying the five lemma to the two short exact sequences.
\end{proof}


\begin{lem}
Let $H\subsetneq F$ be a nontrivial subgroup of a free group and let $f\in F$ be such that $fHf^{-1}\subseteq H$ but $f^n \notin H$ for any $n>0$. Then $H$ is not finitely generated.\label{properconj}\end{lem}

\begin{proof}
Realize $F$ as $\pi_1$ of a bouquet of circles $(B, p)$. Let $(\widehat{B}, \hat{p})$ be a based covering space corresponding to $H$. 
Consider the covering transformation action of $F$ on the vertices of $\widehat{B}$. If $f^m \hat p =f^n \hat p$ then $f^{m-n} \hat p=\hat p$ which implies $f^{m-n}\in H$, which implies $m=n$ by assumption, therefore $\{f^n \hat p, n \in \integers\}$ is infinite.
By non-triviality of $H$, let $\sigma$ be a nontrivial cycle. 
Using that $f^n \notin H$ for any $n>0$, we can find a sequence of elements $\{f^{n_i} \hat{p}\}$ such that the lifts of $\sigma$ starting at $f^{n_i}\hat{p}$ are pairwise disjoint. Hence $\homology_1 (\widehat{B})$ is not finitely generated so $H = \pi_1(\widehat{B})$ is not finitely generated.
\end{proof}

The following corollary is used in the free-by-cyclic case:

\begin{cor}
Let $F$ be a free group and $H\subsetneq F$ a normal subgroup of infinite index, then $H$ is not finitely generated. \label{normalfinite}
\end{cor}

\begin{proof}
Suppose towards a contradiction that $H$ is finitely generated. By Corollary~\ref{cor:splits} there is an element $f\in F-H$ such that $f^n\notin H$ for any $n>0$. Since $H$ is normal, $fHf^{-1} \subseteq H$ so $H$ is not finitely generated by Lemma~\ref{properconj}.
\end{proof}

Finally, we prove the main result of this Section:

\begin{proof}[Proof of Theorem \ref{properHNN}]

We first observe that $[F: \phi_*(F)]=\infty$. Indeed, if $(\widehat{B}, \hat b)$ is a degree $d$ cover of a bouquet of $r\geq2$ circles, then $\chi (\widehat{B})=d\chi(B)=d(1-r)$. Thus $\chi (H) < \chi(F)$ for any proper finite index subgroup $H$ of $F$. So $F$ is not isomorphic to a proper finite index subgroup of itself, and therefore $\phi_*(F)$ is of infinite index.

We inductively show that there exists $f$ such that $\langle \phi^{i}(f) , \ i \in \naturals \rangle = \ast_{i \in \naturals}\langle \phi^i(f) \rangle$. By Corollary~\ref{cor:splits} there exists an $f \in F-\phi(F)$ such that $\langle f, \phi(F) \rangle=\langle f \rangle \ast \phi(F)$. Assuming by induction that for $n \in \naturals$, $\langle f, \phi^{i}(f) , \ i \in \{1, \dots, n\} \rangle = \langle f\rangle \ast_{i\in \{1, \dots, n\}} \langle \phi^i(f) \rangle$. Since $\phi$ is an isomorphism we have $\langle \phi^{i}(f) , \ i \in \{1, \dots, n+1\} \rangle = \ast_{i\in \{1, \dots, n+1\}} \langle \phi^i(f) \rangle$, and since the latter subgroup is included in $\phi(F)$ we get $\langle f, \phi^{i}(f) , \ i \in \{1, \dots, n+1\} \rangle = \langle f \rangle \ast_{i\in \{1, \dots, n+1\}} \langle \phi^i(f) \rangle$, concluding the induction.

In particular Lemma~\ref{freeinf} tells us that $\{t^{i}ft^{-i}, \ i \in \mathbb{Z}\}$ forms a basis for a free subgroup. By Corollary~\ref{infgencor} where $H=\langle f, t \rangle$ and $\rho$ from the HNN structure, $H$ is free of rank $2$. To conclude, $H\cap F$ is conjugated into itself by $t$, and $t^n \notin F$ for any $n>0$, so $H\cap F$ is not finitely generated by Lemma~\ref{properconj}.
\end{proof}



\section{Improved relative train track maps for polynomially growing automorphism}
\label{traintracksection}

Let $\phi : F \rightarrow F$ be an automorphism of a finitely generated free group. Identifying $F$ with $\pi_1(V, v)$ for some finite based graph $V$, the map $\phi$ is induced by a basepoint preserving map $\Phi: V \rightarrow V$. Similarly, outer automorphisms are induced by maps that do not preserve the basepoint.

When $|\phi^n(w)|$ is bounded by a polynomial for each word $w$ in $F$, the map $\phi$ is \textit{polynomially growing}, otherwise it is \textit{exponentially growing}. This is independent of the choice of basis, and could have been expressed using $\Phi$ and closed paths $w$ in $V$.

Bestvina, Feighn and Handel prove that there exists $n\geq1$ such that $\phi^n$ is induced by an \textit{improved relative train track map} $\Psi$ \cite[Thm 5.1.5.]{Bestvina_2000}. Passing to a power $\phi^n$ corresponds to passing to a finite index subgroup, and thus does not change the FGIP discussion.  We now describe the structure of $\Psi$ in the case where the automorphism is polynomially growing and $\rank(F)\geq 1$.
 






\begin{defn}
\label{improved}
Let $\emptyset= V^0 \subsetneq V^1 \subsetneq \cdots \subsetneq V^k = V$ be a filtration of V by subgraphs and let $S^{r}=Cl(V^{r}-V^{r-1})$ be the $r$-th stratum for $r\geq 1$. A \textit{polynomial growth improved relative train track map} $\Psi: V \rightarrow V$ sends vertices to vertices and edges to combinatorial paths, and the following hold:

\begin{enumerate}
    \item Each $V^{r}$ is $\Psi$-invariant. \label{phiinvariant}
    \item  Each $S^{r}$ consists of a single edge $e_r$, and $\Psi(e_r) = e_rP_r$, where $P_r$ is a closed path in $V^{r-1}$ whose initial point is fixed by $\Psi$. \label{polyedge} 
\end{enumerate}
\end{defn}

\begin{rem}
By passing to a power, we may assume that $\Psi$ maps some vertex $v$ to itself, and we regard $v$ as the basepoint of $V$.
\end{rem}

\begin{rem}
If two free group automorphisms $\phi$ and $\phi'$ belong to the same outer class (they differ by composition by an inner automorphism of $F$), then $F\semidirect_\phi \integers \cong F\semidirect_{\phi'} \integers$. Therefore  Theorem~\ref{polynomialcase}, and Corollary~\ref{freebchyp} partition Theorem~\ref{burnsandbrunner}. \end{rem}


\section{Polynomially growing case}
\label{polycase}

In this section we study free-by-cyclic groups that are ascending HNN extensions over polynomially growing improved relative train track maps. To do so we find a $F \times \integers$ subgroup which fails the FGIP. We start by recalling the failure of the FGIP for $F \times \integers$, which was first proved in \cite{Moldavanski68}:





 \begin{thm}
 $F \times \integers$ fails to have the FGIP, when $F$ is a non-cyclic finitely generated free group.
 \label{freecrossint}
 \end{thm}
 
 \begin{proof}
Let $H=\langle (f_0, 0), (f_1, 1) \rangle $ and suppose $J=\langle f_0, f_1 \rangle$ is not a cyclic subgroup of $F$. Since $J$ is non-cyclic, $[f_0, f_1] \neq id$. Hence $H \cap (F\times \{0\})$ is nontrivial, since $[(f_0, 0), (f_1, 1)] \in F \times \{0\}$.


However $[H : H \cap F]= \infty$ because $(f_1, 1)$ has image $1$ under the homomorphism to $\integers$. The result follows from Corollary~\ref{normalfinite}.
\end{proof}





\begin{defn}[Mapping tori]
Let $\Phi : V \rightarrow V$ be a map from a connected graph to itself. The $\textit{mapping torus}$ $M_{\Phi}$ of $\Phi$ is the 2-complex: $M_{\Phi}= V\times [0,1] / 
\{(v,0) \sim (\Phi(v), 1) \ : \ \forall v \in V\}$.
When $\Phi$ is basepoint preserving, $\Phi$ induces a homomorphism $\phi : \pi_1 V \rightarrow \pi_1V$, otherwise we identify $\pi_1(V,p)$ with $\pi_1(V,\Phi(p))$ in the usual way to get a homomorphism $\phi : \pi_1 V \rightarrow \pi_1V$.
\end{defn}


A hierarchy for the group is obtained by means of Definition~\ref{improved}.\eqref{polyedge}. This gives an increasing sequence of mapping tori $M_{\Psi_i}$ where $\Psi_i=\Psi | _{V^{i}}$ and each  stage is obtained from the previous by an HNN extension whose stable letter $e_i$ conjugates $P_it$ to $t$, as described in Definition~\ref{improved}.\eqref{polyedge}. Each inclusion $M_{\Psi_i} \hookrightarrow M_{\Psi}$ is $\pi_1$-injective, so it suffices to show that $M_i$ fails to have the FGIP for some $i$.

We conclude this section with the proof of:

\begin{thm} Any ascending HNN extension of a non-cyclic finite rank free group over an improved relative train track map with no exponential stratum fails to have the FGIP.\label{polynomialcase}\end{thm}

\begin{proof}
Following Definition~\ref{improved}.\eqref{polyedge}, the graph $V$ is formed by adding a sequence of edges $e_1, e_2, \ldots, e_k$. Consider the smallest $r$ where some component $J$ of $V^r$ has $\euler(J)<0$. We show that $\pi_1(M_{\Psi_r})$ fails to have the FGIP. The proof splits into cases according to whether $J - e_r$ is disconnected or connected. In each case we show that there is a subgroup isomorphic to $F_2 \times \integers$ which fails the FGIP by Lemma~\ref{freecrossint}.

When $J - e_r$ is disconnected,  $J - e_r = J_1 \sqcup J_2$ and each $J_i$ deformation retracts to a cycle $c_i$. Note that since $\Psi(J_i)$ intersects $J_i$ by Definition~\ref{improved}.\eqref{polyedge}, and since $J_i$ does not map onto $e_r$ by Definition~\ref{improved}.\eqref{phiinvariant}, each component $J_i$ is mapped to itself. Since $\Psi$ induces an isomorphism, it induces automorphisms of each $\pi_1(c_i)$. Therefore, $\pi_1(M_{\Psi_r})$ splits as $A*_\integers B$ where $A$ and $B$ are isomorphic to $\integers \ltimes \integers$, and the amalgamated $\integers$ corresponds to the stable $\langle t \rangle$ on each side. In particular, we get a $\integers^2 *_\integers \integers^2$ subgroup isomorphic to $F_2\times \integers$ (e.g. use $M_{\Psi_r^2}$).

The second case is when $\euler(J - e_r) = 0$, and $e_r$ starts and ends on $J$. Note that $\Psi(J-e_r)=J-e_r$ and $\Psi_r^2$ induces the identity isomorphism on $\pi_1(J-e_r)$ just like for $J_1$ above. The space $M_{\Psi_r^2}$ is homeomorphic to the space obtained from a torus $T$ by attaching a cylinder along its two boundary circles. These circles correspond to maximal cyclic subgroups in $\pi_1T$. Consequently, $\pi_1(M_{\Psi_r^2})$ is isomorphic to an HNN extension $\integers^2 *_{\alpha^t=\beta}$. This contains a subgroup $\integers^2 *_{\langle\alpha\rangle = \langle \beta \rangle} \integers^2 \cong F_2 \times \integers$.
\end{proof}

\section{Failure of the FGIP for certain relatively hyperbolic groups}
\label{relhyp}

In this section we propose a short argument explaining the failure of the FGIP for certain relatively hyperbolic groups. We then combine this result with a powerful result on relative hyperbolicity of free-by-cyclic groups with exponentially growing automorphism \cite{gautero2007mappingtorus, ghosh2018relative, dahmani2019relative} to give an explanation of the failure of the FGIP for these groups.

\subsection{Ping-pong}
\label{pingpongsubsect}
We recall the relatively hyperbolic generalization of Gromov's application of the ping pong lemma \cite{pingpong} to a hyperbolic group acting as a convergence group on its boundary.

\begin{lem}
Let $G$ be a relatively hyperbolic group. Let $w, t \in G$ be infinite order elements with no common fixed point in $\partial G$. Then there exists $m>0$ such that $\langle w^m, t^m \rangle$ is free of rank $2$.\label{freesubg2}
\end{lem} 

We use the action of $G$ as a convergence group on its boundary $\partial G$. \cite{Bowditch99}

\begin{proof}
 Let $\{a,b\}$ and $\{x, y \}$ be the points stabilized respectively by $w$ and  $t$. Let $N_a , N_b, N_x, N_y$ be pairwise disjoint open neighborhoods of $a,b,x,y$ respectively unless $x=y$ or $a=b$ in which case $N_x=N_y$ or $N_a=N_b$. By compactness we may assume that the closure of these neighborhoods are disjoint.  Since $G$ acts as a convergence group and $w, t$ have $\{a,b\}$ and $\{x, y \}$ as attracting/repelling points respectively. There is a constant $M$ such that $t^m (\bar {N_a} \cup \bar N_b) \subsetneq N_x \cup N_y$ and $w^m (\bar N_x \cup \bar N_y) \subsetneq N_a \cup N_b$ whenever $|m| >M$. Applying the ping pong lemma on $\bar N_a \cup \bar N_b$ and $\bar N_x \cup \bar N_y$  gives us the desired result. 
\end{proof}




\subsection{Relative hyperbolicity and the FGIP}



\begin{thm}
Let $G$ be hyperbolic relative to a collection of subgroups. Let $N \subseteq G$ be a finitely generated subgroup containing a loxodromic element $w$. Suppose $tNt^{-1} \subseteq N$ for some infinite order $t$ with $\langle t \rangle  \cap N = \{1_G\}$.  Then $G$ fails to have the FGIP. In particular, there exists $m$ such that $\langle t^m, w^m \rangle \cap N$ is not finitely generated.
\label{relhypconj2}
\end{thm}

In view of the above, we propose the following definition:

\begin{defn}
A subgroup $N$ of a group $G$ is \textit{ascending} if there exists an infinite order element $t\in G$ with $\langle t \rangle \cap N = \{1_G\}$ and $tNt^{-1} \subseteq N $.
\end{defn}




\begin{proof}[Proof of Theorem \ref{relhypconj2}]
We first show that $t$ and $w$ do not share a fixed point. Let $p$ be a fixed point of $t$. If $w$ fixes $p$, then by \cite[Lem 2.2]{Bowditch99} the subgroup $\langle w, t\rangle$ is either virtually cyclic or consists only of parabolic and elliptic elements. The latter is impossible since $w$ is loxodromic, and the former is impossible because $\langle t \rangle \cap N =\{1_G\}$. 


By Lemma~\ref{freesubg2}, there exists $m \in \integers$ such that $w^m, t^m$ generate a rank~$2$ free subgroup $F$. Observe that $K=N \cap F$ is nontrivial since $w^m \in N \cap F$, and $t K t^{-1} \subseteq K$ but $t^m \notin N$ for $ m > 0$. By Lemma~\ref{properconj}, $K$ is not finitely generated.
\end{proof}

The hyperbolic case simplifies to the following.

\begin{cor}
Let $G$ be hyperbolic, let $N \subseteq G$ be a finitely generated infinite subgroup, and suppose $tNt^{-1} \subseteq N$ for some infinite order $t$ with $\langle t\rangle \cap N$ trivial.  Then $G$ fails to have the FGIP.
\label{hypconj}
\end{cor}

\subsection{Exponential growth free-by-cyclic}
\label{rexporelhyp}
Combining Theorem~\ref{relhypconj2} with recent powerful results on relative hyperbolicity of mapping tori, we obtain an explanation of Theorem~\ref{maintheorem} for exponential growth free-by-cyclic groups.


\begin{cor}
If $\phi: F \rightarrow F$ is an exponentially growing automorphism of a finitely generated free group. Then the free-by-cyclic group $F\semidirect_\phi \integers$ fails to have the FGIP. \label{freebchyp}
\end{cor}

\begin{proof}
 From \cite[Thm~$3.5$]{dahmani2019relative} , $F\semidirect_\phi \integers$ is hyperbolic relative to mapping tori of maximal polynomially growing subgroups. These are the conjugacy classes of finitely many subgroups of form $\langle H, gt \rangle$ where $H<F$ is polynomially growing, $g\in F$ and $t$ is the stable letter of the decomposition.

 Let $w\in F$ be an exponentially growing element. Observe that $w$ is loxodromic in the above relatively hyperbolic structure. Note that the semidirect structure ensures that $\langle t \rangle \cap F = \{1_G\}$. Hence the criterion Theorem~\ref{relhypconj2} applies with $N=F$.
\end{proof}

\textbf{Acknowledgement:} We are grateful to Piotr Przytycki for valuable feedback. We are grateful to Eduardo Martinez Pedroza for ping pong help. We are grateful to the referee for correcting our argument and clarifying the exposition.

\bibliographystyle{plain}
\bibliography{references}

\begin{thebibliography}{10}

\bibitem{BurnsRobBrun}
R.~D\v{z}. Berns and A.~M. Brunner.
\newblock Two remarks on {H}owson's group property.
\newblock {\em Algebra i Logika}, 18(5):513--522, 632, 1979.

\bibitem{Bestvina_2000}
Mladen Bestvina, Mark Feighn, and Michael Handel.
\newblock The tits alternative for out(f n ) i: Dynamics of
  exponentially-growing automorphisms.
\newblock {\em The Annals of Mathematics}, 151(2):517, Mar 2000.

\bibitem{traintracks}
Mladen Bestvina and Michael Handel.
\newblock Train tracks and automorphisms of free groups.
\newblock 1992.

\bibitem{Bowditch99}
B.~H. Bowditch.
\newblock Convergence groups and configuration spaces.
\newblock In {\em Geometric group theory down under ({C}anberra, 1996)}, pages
  23--54. de Gruyter, Berlin, 1999.

\bibitem{dahmani2019relative}
François Dahmani and Ruoyu Li.
\newblock Relative hyperbolicity for automorphisms of free products and free
  groups.
\newblock {\em Journal of Topology and Analysis}, page 1–38, Oct 2020.

\bibitem{gautero2007mappingtorus}
Francois Gautero and Martin Lustig.
\newblock The mapping-torus of a free group automorphism is hyperbolic relative
  to the canonical subgroups of polynomial growth, 2007.

\bibitem{ghosh2018relative}
Pritam Ghosh.
\newblock Relative hyperbolicity of free-by-cyclic extensions, 2018.

\bibitem{MR28836}
Marshall Hall, Jr.
\newblock Subgroups of finite index in free groups.
\newblock {\em Canad. J. Math.}, 1:187--190, 1949.

\bibitem{pingpong}
Pierre Harpe.
\newblock Topics in geometric group theory.
\newblock {\em Bibliovault OAI Repository, the University of Chicago Press}, 01
  2000.

\bibitem{Howson54}
A.~G. Howson.
\newblock On the intersection of finitely generated free groups.
\newblock {\em J. London Math. Soc.}, 29:428--434, 1954.

\bibitem{LiWise2017}
Jiakai Li and Daniel~T. Wise.
\newblock No growth gaps for special cube complexes.
\newblock {\em Math. Proc. Cambridge Philos. Soc.}, pages 1--13.
\newblock Submitted.

\bibitem{Moldavanski68}
D.~I. Moldavanski\u{\i}.
\newblock The intersection of finitely generated subgroups.
\newblock {\em Sibirsk. Mat. \v{Z}.}, 9:1422--1426, 1968.

\bibitem{Short91}
Hamish Short.
\newblock Quasiconvexity and a theorem of {H}owson's.
\newblock In {\'E}.~Ghys, A.~Haefliger, and A.~Verjovsky, editors, {\em Group
  theory from a geometrical viewpoint (Trieste, 1990)}, pages 168--176. World
  Sci. Publishing, River Edge, NJ, 1991.

\end{thebibliography}

~
~
~

Department of Mathematics and Statistics, McGill University, Montreal, Canada
\textit{E-mail adress:} jacob.bamberger@mail.mcgill.ca

~

Department of Mathematics and Statistics, McGill University, Montreal, Canada
\textit{E-mail adress:} daniel.wise@mcgill.ca

\end{document}